\documentclass[oneside,12pt]{amsart}
\usepackage{amssymb, mathrsfs, amscd, mathabx}
\usepackage{tikz, tikz-cd}
\usepackage[all]{xy}
\usepackage{enumerate, amsfonts, amsthm}

\usepackage{hyperref}
\hypersetup{
  pdfborder={0 0 0},
  colorlinks   = true,
  urlcolor     = blue,
  linkcolor    = blue,
  citecolor   = red
}

\usepackage{cleveref}

\usepackage[textsize=tiny]{todonotes}
\setuptodonotes{fancyline, color=blue!30}

\usepackage{microtype} 
\usepackage{lmodern}   

\usepackage[a4paper,margin=1.1in]{geometry}

\newtheorem{theorem}{Theorem}[section]

\newtheorem{cor}[theorem]{Corollary}
\newtheorem{prop}[theorem]{Proposition}
\newtheorem{lem}[theorem]{Lemma}
\newtheorem{rmk}[theorem]{Remark}
\newtheorem{deff}[theorem]{Definition}
\newtheorem{example}[theorem]{Example}
\newtheorem{conjecture}[theorem]{Conjecture}

\newcommand{\del}[2]{{}}

\newcommand{\Z}{\mathbb Z}
\newcommand{\R}{\mathbb R}

\newcommand{\G}{\Gamma}
\newcommand{\bG}{\overline{\Gamma}}

\newcommand{\C}{\mathcal C}

\newcommand{\GG}{G_{\G}}
\newcommand{\KG}{K_{\G}}

\newcommand{\cd}{\mathrm{cd}}
\newcommand{\ugd}{\underline{\mathrm{gd}}}

\newcommand{\vcd}{\mathrm{vcd}}

\newcommand{\uE}{\underline{E}}

\title{Universal Structure of Graph Product Kernels}
\author{Ian J.~Leary}
\address{School of Mathematics, University of Southampton, Southampton SO17~1BJ, UK}
\email{i.j.leary@soton.ac.uk}

\author{Nansen Petrosyan}
\address{School of Mathematics, University of Southampton, Southampton SO17~1BJ, UK}
\email{n.petrosyan@soton.ac.uk}

\thanks{}


\dedicatory{Dedicated to Mike Davis on the occasion of his 75th birthday.}

\date{\today}

\begin{document}

\begin{abstract}
Let $\GG$ be a graph product over a finite simplicial graph $\G$, and let $\KG$ denote the kernel of the canonical homomorphism from $\GG$ to the direct product of its vertex groups.  
It is known that, up to isomorphism, $\KG$ depends only on the underlying graph $\G$ and the cardinalities of the vertex groups.

In this paper we establish a functorial refinement of this fact.   We show that any collection of set maps between the vertex groups naturally induces a homomorphism between the corresponding kernels, and that this construction is functorial. 
Several applications are discussed.
\end{abstract}

\maketitle

\section{Introduction}

Graph products interpolate between free and direct products of groups and have played a central role in geometric group theory. 
Given a simplicial graph $\GG$ with vertex set $V(\G)$ and a collection of vertex groups $\{G_i\}_{i \in V(\G)}$, the \emph{graph product} $\GG$ is obtained by taking the free product of the $G_i$'s and adding commutation relations corresponding to the edges of $\G$. Its structure reflects both the algebra of the vertex groups and the combinatorics of $\G$. 

A canonical epimorphism
\[
\GG \longrightarrow \prod_{i \in V(\G)} G_i
\]
is obtained by mapping each vertex group $G_i$ identically onto its factor and extending multiplicatively.  
The kernel of this map, which we denote by $K_{\G}$, plays a surprisingly geometric role.  
It occurs naturally as the fundamental group of a polyhedral product of classifying spaces, as shown by Panov and Veryovkin~\cite{PanovVeryovkin2016}, and thereby links the algebraic structure of graph products with the topology of moment-angle complexes and related spaces in toric topology.  
In particular, for right-angled Artin and Coxeter groups, these kernels coincide with the fundamental groups of classical moment-angle and real moment-angle complexes.

It was observed by Kim that the isomorphism type of the kernel $\KG$ depends only on the underlying simplicial graph $\G$ and the cardinalities of the vertex groups; see Corollary~15 of~\cite{Kim2012}. 

Our main results strengthen this classification by showing that it arises as a consequence of a much more rigid functorial structure.    
In particular, any collection of set maps between vertex groups induces, in a natural and functorial way, homomorphisms between the corresponding kernels, and this functoriality persists even when passing to more general kernels associated to graph extensions.  

Our approach is fundamentally different from the methods used in~\cite{Kim2012}.  
Rather than relying on the combinatorial and group-theoretic techniques, we work with polyhedral products of classifying spaces, whose fundamental groups realise the kernels under consideration.  
This topological approach is particularly well suited to functorial constructions.

\begin{theorem}\label{thm:kernel_functoriality}
Let $\G$ be a finite simplicial graph, and let $\{G_i\}_{i \in V(\G)}$ and $\{G_i'\}_{i \in V(\G)}$ be two families of groups. 
Let $G_{\G}$ and $G_{\G}'$ be the corresponding graph products, and denote
\[
K_{\G} := \ker\Big(G_{\G} \longrightarrow \prod_{i \in V(\G)} G_i\Big), 
\quad 
K_{\G}' := \ker\Big(G_{\G}' \longrightarrow \prod_{i \in V(\G)} G_i'\Big).
\] 
\noindent Suppose for each $i \in V(\G)$ there is a set map
$f_i : G_i \longrightarrow G_i'.$
Then these maps induce a  homomorphism $\Phi_\Gamma: K_{\G} \longrightarrow K_{\G}',$
which defines a functor from the category of sequences of sets $\{G_i\}_{i \in V(\G)}$ (with set maps as morphisms) to the category of groups, sending $\{G_i\}$ to $K_{\G}$ and $\{f_i\}$ to the induced map $\Phi_\Gamma$.  
\end{theorem}

We call $\Phi_\Gamma$ the {\it homomorphism induced by the collection $\{f_i\}$}. In Proposition \ref{prop:explicit_homomorphism}, we give a simple and explicit formula for  $\Phi_\Gamma$, expressed directly in terms of syllables of the words in $\KG$.

As an application of Theorem \ref{thm:kernel_functoriality}, we obtain the following more general result.

\begin{theorem}\label{thm:general_kernel_functoriality}
Let $\G$ be a  finite simplicial graph and $\bG$ be a simplicial graph obtained from $\G$ by adding  some edges. Let $\{G_i\}_{i \in V(\G)}$ and $\{G_i'\}_{i \in V(\G)}$ be two families of groups. 
Denote,
\[
K_{\G, \bG} := \ker\Big(G_{\G} \longrightarrow G_{\bG}\Big), 
\quad 
K_{\G, \bG}' := \ker\Big(G_{\G}' \longrightarrow G'_{\bG}\Big).
\] 
\noindent Suppose for each $i \in V(\G)$ there is a set map
$f_i : G_i \longrightarrow G_i'.$
Then the  homomorphism $\Phi_{\G}: K_{\G} \longrightarrow K_{\G}'$ restricts to a homomorphism $\Phi_{\G, \bG}: K_{\G, \bG} \longrightarrow K_{\G, \bG}'$. This 
defines a functor from the category of sequences of sets $\{G_i\}_{i \in V(\G)}$ (with set maps as morphisms) to the category of groups, sending $\{G_i\}$ to $K_{\G, \bG}$ and $\{f_i\}$ to the induced map $\Phi_{\G, \bG}$.  
\end{theorem}

Note that when $\bG$ is the complete graph on $V(\G)$,i.e., a single simplex with vertex set $V(\G)$, then $K_{\G, \bG}=K_{\G}$. So, Theorem \ref{thm:general_kernel_functoriality} is indeed a generalization of Theorem \ref{thm:kernel_functoriality}. As a further generalization, one may allow a simplicial map of graphs
$\psi\colon \G_1 \to \G_2$ together with an extension
$\bar\psi\colon \bG_1 \to \bG_2$, families of groups
$\{G_i\}_{i \in V(\G_1)}$ and $\{G_i'\}_{i \in V(\G_2)}$, and set maps
$f_i \colon G_i \longrightarrow G'_{\psi(i)}$.
This more general situation is treated in
Theorem \ref{thm:more_general_kernel_functoriality}. 

As an immediate application of Theorem \ref{thm:kernel_functoriality}, we obtain the following result of \cite{Kim2012}.

\begin{cor}[{\cite[Corollary 15(2)]{Kim2012}}]\label{intro_cor_2}
Let $\G$ be a finite simplicial graph and $\bG$ be a simplicial graph obtained from $\G$ by adding  some edges. Let $\{G_i\}_{i \in V(\G)}$ and $\{G_i'\}_{i \in V(\G)}$ be two families of groups such that, for each $i$, $|G_i|=|G_i'|$. Then, the corresponding graph product kernels are naturally isomorphic, $K_{\G, \bG}\cong K'_{\G, \bG}$. 
\end{cor}

 In fact, under the normal length metric on $\GG$, the above isomorphism  is an isometry (see Corollary \ref{cor:isometric_embedding}).

\begin{cor}\label{intro_cor}
Let $\G$ be a finite simplicial graph and let $\GG$ be the graph product of groups $\{G_i\}_{i \in V(\G)}$. 
\begin{enumerate}[(i)]
    \item If each $G_i$ is infinite countable, then $\KG$ is isomorphic to the commutator subgroup of the right-angled Artin group $A_{\G}$ on $\G$.
    
    \item If each $G_i$ is countable, then $\KG$ is a retract of the commutator subgroup of the right-angled Artin group $A_{\G}$ on $\G$. 
    
    \item If each $G_i$ is nontrivial, then the commutator subgroup of the right-angled Coxeter group $W_{\G}$ on $\G$ is a retract of $\KG$.
\end{enumerate}
\end{cor}

Corollary \ref{intro_cor} (i) and (ii) is a strengthening of Corollary 14 of \cite{Kim2012}. Part (iii) has also been established in Lemma 2.7 of \cite{Li_2023}.

Corollary \ref{intro_cor} (ii) is useful in investigating the hereditary properties of graph products.
For instance, it is well-known that finitely generated right-angled Artin groups are poly-free (e.g. \cite{HerSun07}). One may extend this result as follows:

\begin{cor}\label{intro_cor:poly-free}
    Let $\G$ be a finite simplicial graph and let $\GG$ be the graph product of countable groups $\{G_i\}_{i \in V(\G)}$. Then $\KG$ is poly-free. In addition, if each vertex group is poly-free, then $\GG$ is poly-free.
\end{cor}

Let  $\C$ be a class of groups satisfying the following  properties:

\begin{enumerate}[(P1)]
    \item Every subgroup of any finitely generated right-angled Artin group is in $\C$;
    \item If $B$ is a group that has a normal subgroup $A\in \C$ such that $B/A \in \C$, then $B\in \C$.
    
\end{enumerate}

\begin{cor}\label{intro_general_property} Let $\C$ be a class of groups satisfying properties $(P1)$ and $(P2)$.
Let $\G$ be a finite simplicial graph and let $\GG$ be the graph product of countable groups $\{G_i\}_{i \in V(\G)}$. If each $G_i\in \C$, then $\GG\in \C$.
\end{cor}

There is an application of Corollary \ref{intro_general_property} to the Baum--Connes conjecture.

\begin{deff}[Baum--Connes Conjecture with Coefficients and Finite Wreath Products, \cite{NP25}]\rm
A countable discrete group $G$ is said to satisfy the \emph{Baum--Connes conjecture with coefficients with finite wreath products} (abbreviated \emph{BCC with finite wreath products}) if, for every finite group $F$, the wreath product $G \wr F$ satisfies the Baum--Connes conjecture with coefficients.
\end{deff}

In particular, any group satisfying this conjecture also satisfies the usual Baum--Connes conjecture with coefficients.  
An advantage of this strengthened version of the conjecture is that it is stable under group extensions \cite[Theorem~4.3(6)]{NP25}.

It is known that the Baum--Connes conjecture with coefficients is closed under graph 
products, since it is stable under both finite direct products and amalgamated products 
along subgroups \cite{OO2001Extensions, OO2001Trees}.

\begin{cor}\label{intro_cor:Baum--Connes}
Let $\G$ be a countable simplicial graph  and let $\GG$ be the graph product of countable groups $\{G_i\}_{i \in V(\G)}$.  
If each vertex group $G_i$ satisfies BCC with finite wreath products, then $\GG$ also satisfies BCC with finite wreath products.
\end{cor}

\subsection*{Acknowledgments.} We would like thank Anthony Genevois for making us aware of his unpublished work related to our main results and also for bringing to our attention the results of Kim, in particular, \cite[Corollary 15]{Kim2012}. We thank Kevin Li for pointing out Lemma 2.7 in \cite{Li_2023}. We also thank Ashot Minasyan for useful discussions.

\section{Preliminaries}\label{sec:prelims}

\subsection{Graph products of groups} Let $\G$ be a simplicial graph. For each vertex $i \in V(\G)$, let $G_i$ be a nontrivial group.  
The \emph{graph product} of the family $\{G_i\}_{i \in V(\G)}$, denoted $G_L$, is the quotient of the free product $\bigast_{i \in V(\G)} G_i$ by the relations
\[
[g_i, g_j] = 1 \quad \text{whenever $\{i,j\}$ is an edge of $\G$.}
\]

Two important special cases:
\begin{itemize}
    \item If each $G_i = \Z$, then $\GG$ is the \emph{right-angled Artin group} $A_{\G}$.
    \item If each $G_i = C_2$, then $\GG$ is the \emph{right-angled Coxeter group} $C_{\G}$.
\end{itemize}

There is a canonical epimorphism
\[
\phi_{\G}: \GG \longrightarrow \prod_{i \in V(\G)} G_i
\]
sending each generator of $G_i$ to its image in the corresponding factor. Denote its kernel by
$\KG = \ker(\phi_{\G})$.

The graph $\G$ determines a unique flag complex $L=L(\G)$, called the {\it clique complex of $\G$}, obtained by filling in a simplex for every complete subgraph of $\G$.  
Since this association is canonical, we will use the notations $\G$ and $L$ interchangeably.  
Accordingly, we shall write $G_L$ and $K_L$ in place of $\GG$ and $\KG$ when convenient, and similarly for the corresponding constructions throughout the paper.

\subsection{Polyhedral products}

Polyhedral products provide a topological model connecting graph products to classifying spaces and their kernels.  
We briefly recall the standard definitions and some basic properties.

\begin{deff}{\rm
Let $K$ be a simplicial complex on the finite vertex set  $[m] = \{1,\dots,m\}$, and let $(\underline X, \underline A):=\{(X_i,A_i)\}_{i=1}^m$ be a sequence of CW-pairs.  
The \emph{polyhedral product} associated to $K$ and $(\underline X, \underline A)$ is the subcomplex
\[
(\underline X,\underline A)^K := \bigcup_{\sigma \in K} \prod_{i=1}^m Y_i^\sigma \subseteq \prod_{i=1}^m X_i,
\]
where
\[
Y_i^\sigma = 
\begin{cases}
X_i, & i \in \sigma, \\
A_i, & i \notin \sigma.
\end{cases}
\]
}
\end{deff}

\begin{rmk}\rm
If all $A_i$ are points, we often write $\underline X^K$ to denote $(\underline X, {\bf \ast})^K$ for simplicity.  
Polyhedral products interpolate between Cartesian products and wedge sums: for the full simplex, $\underline X^K = \prod_i X_i$, and for the discrete complex (only vertices), $\underline X^K = \bigvee_i X_i$.
\end{rmk}

\begin{prop}[Functoriality of polyhedral products]\label{prop:functorial}
Let $K_1$ and $K_2$ be simplicial complexes, and let 
$\{(X_i,A_i)\}_{i \in V(K_1)}$ and $\{(X_j',A_j')\}_{j \in V(K_2)}$ be sequences of CW-pairs.  
Suppose we are given
\begin{itemize}
    \item[-] a simplicial map $\varphi: K_1 \to K_2$, and
    \item[-] a collection of maps of pairs $f_i:(X_i,A_i) \to (X_{\varphi(i)}',A_{\varphi(i)}')$ for each $i \in V(K_1)$.
\end{itemize}

Then there is an induced cellular map
\[
f^\varphi : (\underline{X},\underline{A})^{K_1} \longrightarrow (\underline{X}',\underline{A}')^{K_2}.
\]

Moreover:
\begin{enumerate}[(i)]
    \item If for each $i$, $f_i$ is homotopic to $g_i$ relative to $A_i$, then the induced maps satisfy
    \[
    f^\varphi \simeq g^\varphi \colon (\underline{X},\underline{A})^{K_1} \longrightarrow (\underline{X}',\underline{A}')^{K_2}.
    \]
    \item The map $f^\varphi$ is functorial both in the CW-pairs and in the simplicial complex.
\end{enumerate}
\end{prop}

\subsection{Classifying spaces and polyhedral products}

Here we state some standard facts about classifying spaces of groups and their connection to polyhedral products. We refer to Section 3 of \cite{PanovVeryovkin2016} for further details. 

 A {\it $G$-CW-complex} is a CW-complex $X$ equipped with a $G$-action that is compatible with the cell structure, in the sense of L\"uck \cite[Def.~1.25]{luck2002l2}.  This notion is equivalent to that of an admissible $G$-complex in the terminology of Brown \cite[Chap.~IX, §10]{brown1982cohomology}, characterised by the property that any element of $G$ which stabilises a cell $\sigma \subseteq X$ must fix $\sigma$ pointwise.  Throughout, we will assume that all $G$-actions are of this form.

We say that $X$ is a \textit{free $G$-CW complex} if the action $G$ on $X$ is free. 
By $EG$ we mean a contractible free $G$-CW complex, and by $BG$ we mean the quotient $EG/G$ called a {\it classifying space} of $G$. If we do not specify the choice of $EG$, then we simply choose an arbitrary contractible free $G$-CW complex as $EG$. Such a complex always exists and any two choices for $EG$ are $G$-homotopy equivalent \cite{luck05_survey}. Equivalently, $BG$ is a $K(G,1)$-complex, so $\pi_1(BG) = G$ and $\pi_i(BG) = 0$ for $i\ne 1$.

Now let $L$ be a finite complex with vertices $V(L) = \{1, \dots, m\}$, and let $\{G_i\}_{i\in V(L)}$ be a collection of discrete groups.  Denote $B\underline  G=\{(BG_i,\ast )\}_{i=1}^m$ and $(E\underline  G, \underline  G)=\{(EG_i, G_i)\}_{i=1}^m$. Recall,
\[(B\underline G)^L = \bigcup_{\sigma \in L} \prod_{i \in V(L)} Y_i, \quad Y_i = 
\begin{cases} 
BG_i & i \in \sigma,\\
\ast & i \notin \sigma.
\end{cases}
\]
\noindent Similarly,  the polyhedral product  associated to  $(E\underline  G, \underline  G)$ is given by:
\[
(E\underline G,\underline G)^L := \bigcup_{\sigma \in L} \prod_{i \in V(L)} Y_i, \quad Y_i = 
\begin{cases} 
EG_i & i \in \sigma,\\
G_i & i \notin \sigma,
\end{cases}
\]
where each $G_i$ is included in $EG_i$ as the fiber over the basepoint of $BG_i$.
The fundamental group of $(B\underline{G})^L$ is the graph product $G_\Gamma$ for 
$\Gamma$ equal to the 1-skeleton of $L$.  More can be said in the case when $L$ 
is flag.  

\begin{prop}[Panov--Veryovkin~\cite{PanovVeryovkin2016}]\label{prop:aspherical} In the case when $L$ 
is a flag complex, $(B\underline{G})^L$ is aspherical, and so it is a classifying 
space for the group $G_\Gamma$ for $\Gamma$ the 1-skeleton of $L$.  
\end{prop}

To give a sketch proof of this we require a lemma.  For $W$ a subset of the vertex
set $V$ of $L$, the full subcomplex $L_W$ is the complex whose simplices are the 
simplices of $L$ that are contained in $W$.  

\begin{lem}[{Panov--Veryovkin \cite[Prop.~2.2]{PanovVeryovkin2016}}]\label{lemm:retract} 
For $K$ a full subcomplex of $L$, $(B\underline{G})^K$ is a retract 
of $(B\underline{G})^L$.  
\end{lem}

\begin{proof} Let $K$ be the full subcomplex on $W\subseteq V$, and define a 
new sequence $Y_i$ of based spaces by 
\[Y_i=\begin{cases}
BG_i& \hbox{if $i\in W$,}\cr 
*& \hbox{if $i\notin W$.} \cr \end{cases}\]
There is a natural homeomorphism between $(B\underline{G})^K$ and $\underline{Y}^L$, 
and since each $Y_i$ is a based retract of $BG_i$, $\underline{Y}^L$ is a retract 
of $(B\underline{G})^L$.  
\end{proof}

\begin{proof}[Proof of Proposition~\ref{prop:aspherical}]
  We prove Proposition~\ref{prop:aspherical} by induction on $L$.  In the case 
  when $L$ is the $(m-1)$-simplex, $(B\underline{G})^L=\prod_iBG_i$, so it is 
  aspherical in this case.  If $L$ is a flag complex which is not a simplex, then 
  there exist $v_1,v_2\in V$ that are not joined by an edge.  In this case, for 
  $i=1,2$ let $L_i$ be the full subcomplex of $L$ with vertex set $V-\{v_i\}$ 
  and let $L_3$ be the intersection $L_1\cap L_2$.  Each $L_i$ is a flag complex
  with fewer vertices than $L$, and so by induction each is aspherical.  By 
  Lemma~\ref{lemm:retract} the fundamental group of $(B\underline{G})^{L_3}$ 
  maps injectively to the fundamental group of $(B\underline{G})^{L_i}$ for 
  $i=1,2$.  If we write $\Gamma_i$ for the 1-skeleton of $L_i$, then $G_\Gamma$ 
  is isomorphic to the free product with amalgamation 
  \[G_\Gamma=G_{\Gamma_1}*_{G_{\Gamma_3}}G_{\Gamma_2},\]
  and a classifying space for this group can be obtained as the homotopy 
  pushout of the given maps $(B\underline{G})^{L_3}\rightarrow (B\underline{G})^{L_i}$ 
  for $i=1,2$.  Since these maps are the inclusions of subcomplexes, they are 
  cofibrant, and so the homotopy pushout is homotopy equivalent to the pushout, 
  which is the union, i.e., the space $(B\underline{G})^L$.  Thus this space is 
  a classifying space for $G_\Gamma$ and is aspherical as required.  
\end{proof}

By Proposition~\ref{prop:functorial}, the quotient map associated to the universal covering $EG_i \to BG_i$ for each $i$ induces a map of pairs 
$(EG_i, G_i) \longrightarrow (BG_i, *)$,
which in turn gives a map of polyhedral products
$$(E\underline G,\underline G)^L \longrightarrow (B\underline G)^L.$$
Similarly, the inclusion of the simplicial complex $L$ into the full simplex on $V(L)$ induces a natural map
$$(B\underline G)^L \longrightarrow \prod_{i \in V(L)} BG_i.$$

The key structural result, due to Panov and Veryovkin~\cite{PanovVeryovkin2016}, asserts that these maps form
 a homotopy fibration.

\begin{prop}[{Panov--Veryovkin, \cite[Prop.\,3.1]{PanovVeryovkin2016}}]\label{prop:PV}
The sequence of canonical maps
\[
(E\underline G,\underline G)^L \longrightarrow (B\underline G)^L \longrightarrow \prod_{i \in V(L)} BG_i
\]
is a homotopy fibration. In particular, the induced map on fundamental groups
\[
\pi_1((E\underline G,\underline G)^L) \longrightarrow \pi_1((B\underline G)^L) = G_L
\]
is injective, and its image is exactly the kernel $K_L$.  In the case when $L$ is 
flag, it follows that $(E\underline{G},\underline{G})^L$ is a classifying space 
for $K_L$.  
\end{prop}

\begin{proof} 
In fact, we show that the map $(E\underline{G},\underline{G})^L\rightarrow 
(B\underline{G})^L$ is a regular covering map, with $\prod_i G_i$ as its group
of deck transformations.  
In the case when $L$ is the $(m-1)$-simplex, this is just the familiar fact that 
$\prod_i EG_i\rightarrow \prod_iBG_i$ is the universal covering.  In the general case, 
the inclusion $(B\underline{G})^L\rightarrow \prod_i BG_i$ is surjective on 
fundamental groups, and so pulling back the universal covering along this 
map gives a connected covering space of $(B\underline{G})^L$ with $\prod_iG_i$ 
as its group of deck transformations.  Since $(B\underline{G})^L$ is a 
subcomplex of $\prod_iBG_i$, it follows that the pullback is a subcomplex 
of $\prod_iEG_i$.  Moreover, the pullback is isomorphic to 
$(E\underline{G},\underline{G})^L$.  
\end{proof}

Two classical instances illustrate this general construction.  
When each $G_i = \Z$, the pair $(E\Z, \Z)$ is homotopy equivalent to $(\R, \Z)$, yielding the familiar fibration associated to right-angled Artin groups with a homotopy fiber $(E\underline A, \underline A)^L$.  
When each $G_i =C_2$, the pair $(EC_2, C_2)$ is homotopy equivalent to $(D^1, S^0)$, recovering the classical fibration for right-angled Coxeter groups with a homotopy fiber $(E\underline C_2, \underline C_2)^L$.   
For reference, recall that $B\Z \simeq S^1$ with universal cover $\R \to S^1$, and $BC_2 \simeq \R P^\infty$ with universal cover $S^\infty \to \R P^\infty$.

There is a converse to Proposition~\ref{prop:aspherical}, which can be proved
readily by quoting our main result.  

\begin{prop}[Panov--Veryovkin~\cite{PanovVeryovkin2016}]  
If $L$ is not a flag complex and each $G_i$ is non-trivial, then 
$(B\underline{G})^L$ is not aspherical.  
\end{prop}

\begin{proof}
By Lemma~\ref{lemm:retract}, it suffices to consider the case when $L$ is 
minimal non-flag, i.e. $L$ is the boundary of an $(m-1)$-simplex 
for some $m\geq 3$.  
By Proposition~\ref{prop:PV}, it suffices to show instead that provided that 
each $G_i$ is non-trivial, $(E\underline{G},\underline{G})^L$ is not aspherical.  
By Corollary~\ref{intro_cor}, it suffices to consider the case when each $G_i$ 
has order two.  In the case when $G$ has order two, the pair $(EG,G)$ is homotopy
equivalent to the pair $(D^1,S^0)$.  Hence in the case when $L$ is the boundary of 
an $(m-1)$-simplex, $(D^1,S^0)^L$ is equal to the boundary of an $m$-cube, a 
space homotopy equivalent to $S^{m-1}$.  
\end{proof}

\section{Proof of Main Result}

\begin{lem}\label{lem:homotopy_relative}
Let $(X,A)$ and $(Y,B)$ be CW-pairs with  $Y$ contractible.  
If $f_1, f_2 : (X,A) \to (Y,B)$ are maps of pairs agreeing on $A$, then $f_1$ and $f_2$ are homotopic relative to $A$.
\end{lem}

\begin{proof}
Let $f_1, f_2: (X,A) \to (Y,B)$ be maps of pairs agreeing on $A$.  
Consider the CW-pair 
\[
(X \times [0,1],\, A \times [0,1] \cup X \times \{0,1\}),
\] 
and the problem of extending the map defined on the subcomplex
\[
A \times [0,1] \cup X \times \{0,1\} \longrightarrow Y
\]
(which is given by $f_1|_A = f_2|_A$ on $A \times [0,1]$ and by $f_1, f_2$ on $X \times \{0,1\}$)
to a map on all of $X \times [0,1]$.  

Since $Y$ is contractible, all relative cohomology groups 
\[
H^{n+1}(X \times [0,1],\, A \times [0,1] \cup X \times \{0,1\}; \pi_n(Y))
\] 
vanish for $n \ge 1$. By the Cellular Extension Theorem \cite[Corollary 4.73]{hatcher2002algebraic},  there are therefore no obstructions to extending the map.  
Hence, a homotopy $H: X \times [0,1] \to Y$ exists relative to $A$, showing that $f_1$ and $f_2$ are homotopic relative to $A$.
\end{proof}

\begin{proof}[Proof of Theorem \ref{thm:kernel_functoriality}]
Let $L$ be the flag complex determined by $\G$.  
Applying the Cellular Extension Theorem, we extend each $f_i:G_i  \to G_i'$ to a cellular map of pairs 
\[
(\bar f_i, f_i): (EG_i, G_i) \longrightarrow (E G_i', G_i').
\]  
By Lemma \ref{lem:homotopy_relative}, any two such extensions are homotopic relative $G_i$.  
By Proposition \ref{prop:functorial}, this induces a cellular map 
\[
f:(E\underline{G},\underline{G})^L \longrightarrow (E\underline{G}', \underline{G}')^L
\] 
which is well-defined up to homotopy: different choices of extensions yield homotopic maps.  Passing to the fundametal groups, we obtain the desired group homomorphism $\Phi: \KG\to \KG'$.

Functoriality of this construction also follows from Lemma \ref{lem:homotopy_relative}.
\end{proof}

\begin{proof}[Proof of Theorem \ref{thm:general_kernel_functoriality}] Let $\Delta$ be the complete graph on $V(\G)$. The simplicial maps $\G\to \bG\to \Delta$ yield the following commutative diagram with exact rows and columns: 

\[
\begin{tikzcd}
   K_{\G, \bG} \arrow[r, equals] \arrow[d, tail] 
      & K_{\G, \bG}  \arrow[d, tail] 
              \\
   K_{\G} \arrow[r, tail] \arrow[d, two heads] 
      & G_{\G} \arrow[r, two heads] \arrow[d, two heads] 
          & \prod_{i \in V(\G)} G_i  \arrow[d, equals] 
              \\
   K_{\bG} \arrow[r, tail] 
      & G_{\bG} \arrow[r, two heads] 
          & \prod_{i \in V(\G)} G_i 
\end{tikzcd}
\]

\noindent In particular, we obtain the short exact sequence of the kernels $K_{\G,\bG} \rightarrowtail K_{\G} \twoheadrightarrow K_{\bG}$. Similarly, there is a short exact sequence
$K'_{\G,\bG} \rightarrowtail K'_{\G} \twoheadrightarrow K'_{\bG}$. 

Applying Theorem \ref{thm:kernel_functoriality} (observe that the construction of the homomorphism in Theorem \ref{thm:kernel_functoriality} is functorial in $\G$), we obtain the commutative diagram:
\[
\begin{tikzcd}
K_{\G,\bG} \arrow[r, tail] \arrow[d, "\Phi_{\G, \bG}"]
  & K_{\G} \arrow[r, two heads] \arrow[d, "\Phi_{\G}"]
      & K_{\bG} \arrow[d, "\Phi_{\bG}"] \\
K'_{\G,\bG} \arrow[r, tail]
  & K'_{\G} \arrow[r, two heads]
      & K'_{\bG}
\end{tikzcd}
\]
where $\Phi_{\G, \bG}$ is the restriction of $\Phi_{\G}$. The functoriality assertion is immediate from Theorem \ref{thm:kernel_functoriality}.
\end{proof}

Next, we present an even more general version of Theorem \ref{thm:kernel_functoriality}.

\begin{theorem}\label{thm:more_general_kernel_functoriality} Let $\G_1, \G_2$ be a  simplicial graphs and $\bG_1, \bG_2$ be a simplicial graph obtained from $\G$ by adding  some edges. Suppose there is a simplical map  $\psi: \G_1 \to \G_2$ which extends to a simplicial map $\bar\psi:\bG_1 \to \bG_2$. Let $\{G_i\}_{i \in V(\G_1)}$ and $\{G_i'\}_{i \in V(\G_2)}$ be two families of groups. 
Denote,
\[
K_{\G_1, \bG_1} := \ker\Big(G_{\G_1} \longrightarrow G_{\bG_1}\Big), 
\quad 
K_{\G_2, \bG_2}' := \ker\Big(G_{\G_2}' \longrightarrow G'_{\bG_2}\Big).
\] 
\noindent Suppose for each $i \in V(\G_1)$ there is a set map
$f_i : G_i \longrightarrow G_{\psi(i)}'.$
Then these maps induce a  homomorphism $\Phi_{\psi}: K_{\G_1} \longrightarrow K_{\G_2}'$ which restricts to a homomorphism $\Phi_{\psi, \bar\psi}: K_{\G_1, \bG_1} \longrightarrow K_{\G_2, \bG_2}'$. This 
defines a functor from the category whose objects are graph pairs $(\G, \bG)$ with vertex groups $\{G_i\}_{i \in V(\G)}$ and morphisms are pairs of simplicial maps $(\psi, \bar\psi)$ together with set maps $\{f_i\}_{i \in V(\G_1)}$, sending each object to $K_{\G, \bG}$ and each morphism to the induced homomorphism $\Phi_{\psi, \bar\psi}$.  
    
\end{theorem}

\begin{proof} The proof can be carried out completely analogously to the proofs of Theorems \ref{thm:kernel_functoriality} and \ref{thm:general_kernel_functoriality}. First, by Proposition \ref{prop:PV}, one obtains a homorphism $\Phi_{\psi}: K_{\G_1} \longrightarrow K_{\G_2}'$. This construction is functorial on both the associated simplicial graph and the vertex groups. Proceeding as in the proof of Theorem \ref{thm:general_kernel_functoriality}, one obtains the commutative diagram:
\[
\begin{tikzcd}
K_{\G_1,\bG_1} \arrow[r, tail] \arrow[d, "\Phi_{\psi, \bar\psi}"]
  & K_{\G_1} \arrow[r, two heads] \arrow[d, "\Phi_{\psi}"]
      & K_{\bG_1} \arrow[d, "\Phi_{\bar\psi}"] \\
K'_{\G_2,\bG_2} \arrow[r, tail]
  & K'_{\G_2} \arrow[r, two heads]
      & K'_{\bG_2}
\end{tikzcd}
\]
Since the construction of $\Phi_{\psi}$ is functorial, it follows that the construction of $\Phi_{\psi, \bar\psi}$ is as well.
\end{proof}

\section{Explicit formula for the homomorphism}

Next we provide a detailed description of how the functoriality of the kernel is realized. 

Let $\G$ be a finite simplicial graph and let $\{G_i\}_{i \in V(\G)}$ be a collection of vertex groups. 
An element $g \in G_\G$ can be expressed as a product
\[
w = g_1 g_2 \cdots g_n,
\]
where each $g_j$ is an element of some vertex group $G_i$. 
We call $w$ a \emph{word representing $g$}, and the individual $g_j$ are called the \emph{syllables} of $w$.  

\begin{prop}[Explicit formula for $\Phi_\Gamma$]\label{prop:explicit_homomorphism}
Let $\G$ be a finite simplicial graph, and let $\{G_i\}_{i \in V(\G)}$ and $\{G_i'\}_{i \in V(\G)}$ be two families of groups. 
Given a collection of set maps $f_i : G_i \to G_i'$, then the induced homomorphism
$\Phi_\Gamma: K_{\G} \longrightarrow K_{\G}'$
is given by the following formula:

Given $g\in \GG$, let $w \in K_{\G}$ be a word in the vertex groups $\{G_i\}_{i\in V(\G)}$ representing $g$. 
For each $i \in V(\G)$, let $m_i$ be the number of syllables of $w$ belonging to $G_i$, and, for each $1 \le j \le m_i$, let 
$g_{ij} \in G_i$ be the $j$-th occurring syllable from $G_i$ in $w$ (counted from left to right), and define $g_{i0}=e_i$.  

Then, $\Phi_\Gamma(g)$ is the ordered product of syllables $g'_{ij} \in G_i'$ appearing in the same syllable positions, where 
\[
g'_{ij} = f_i(g_{i1} \cdots g_{ij-1})^{-1}\cdot f_i(g_{i1} \cdots g_{ij}),\; \; \forall 1\leq j\leq m_i, \mbox{ and } g'_{i0}=e'_i.
\]

\end{prop}

We will give a topological proof of Proposition~\ref{prop:explicit_homomorphism}, but 
first we need to introduce some more notation.  For groups $G_i$ as in the statement, 
let $S$ be the disjoint union of the sets $G_i-\{e_i\}$, i.e., the collection of all 
syllables, and let $S^*$ denote the collection of all words in the elements of $S$, 
including an empty word denoted $e$.  Define $S'$ and $S'^*$ similarly in terms of 
the groups $G'_i$.  Given set maps $f_i:G_i\rightarrow G'_i$ as 
in the statement, denote by $\mathbf{f}$ the set map $\mathbf{f}:\prod_i G_i\rightarrow 
\prod_iG'_i$ defined by 
\[\mathbf{f}(g_1,\ldots,g_m)=(f_1(g_1),\ldots,f_m(g_m)).\]  

By a slight abuse of notation, we identify elements of $S$ with their images in $G_\Gamma$ 
for each graph $\Gamma$ with vertex set $\{1,\ldots,m\}$, and thus obtain a surjective map 
$S^*\rightarrow G_\Gamma$.  In particular, this identifies 
$S$ with the set of elements of $\prod_iG_i$ 
that are not equal to the identity in exactly one coordinate: if $s\in G_i\cap S$, 
this identifies $s$ with $(e_1,\ldots,e_{i-1},s,e_{i+1},\ldots,e_m)$.  

\begin{prop} With notation as described above, for any $\bold{g}\in \prod_iG_i$ 
and any $s\in S$, $\mathbf{f}(\mathbf{g})^{-1}\mathbf{f}(\mathbf{g}s)$ is either 
an element of $S'$ or is the identity element of $\prod_iG'_i$.  There 
is a unique map $\Phi:S^*\rightarrow S'^*$ 
defined recursively as follows: $\Phi(e)=e'$, i.e., $\Phi$ sends the empty word to the 
empty word, and for any $s\in S$ and any $w\in S^*$, 
\[\Phi(ws)= \begin{cases} 
\Phi(w) &  \text{if}\,\,\,\, \mathbf{f}(w)=\mathbf{f}(ws),\\
\Phi(w) (\mathbf{f}(w)^{-1}\mathbf{f}(ws)) & \text{if}\,\,\,\,\mathbf{f}(w)\neq \mathbf{f}(ws).  
\end{cases}\]
For each graph $\Gamma$ with vertex set $\{1,\ldots,m\}$ 
the map $\Phi$ passes to a well-defined set map of groups $\Phi_\Gamma:G_\Gamma\rightarrow 
G'_\Gamma$, and this restricts to a group homomorphism $\Phi_\Gamma:K_\Gamma\rightarrow K'_\Gamma$.  Each of these maps and homomorphisms is functorial in the set maps $f_i$.  
\end{prop}

\begin{proof} 
Let $L=L(\Gamma)$ be the clique complex of the graph $\Gamma$, and let $BG_i$ be a 
classifying space for the group $G_i$ with a fixed basepoint $*$.  As described in 
Section~\ref{sec:prelims}, the polyhedral product $(B\underline{G})^L$ is a 
classifying space for the graph product $G_\Gamma$, and by Proposition~\ref{prop:PV} 
the polyhedral product $(E\underline{G},\underline{G})^L$ is a classifying space for 
the subgroup $K_L$, where the subspace $G_i\subseteq EG_i$ is the set of lifts of the
basepoint.  Moreover, the maps $EG_i\rightarrow BG_i$ induce the covering map 
$BK_L\rightarrow BG_L$.  

By definition, $G_\Gamma$ is identified with the set of homotopy classes of based loops in 
$(B\underline{G})^L$.  For any fixed basepoint in $(E\underline{G},\underline{G})^L$ 
this gives an identification of the set $G_\Gamma$ with the set of homotopy classes 
of paths in $(E\underline{G},\underline{G})^L$ that start the chosen lift of the 
basepoint and end at some possibly different lift of the basepoint.  Under this 
identification the subgroup $K_\Gamma$ is identified with the subset of homotopy classes 
of closed loops.  In this way, the groups that arise in the statement are realized as
sets of homotopy classes of paths in $(E\underline{G},\underline{G})^L$.  Our method 
of proof involves replacing the pair $(EG_i,G_i)$ by a homotopy equivalent pair in 
such a way that the collection $S^*$ of all words in $S$ is realized as the collection 
of edge paths in the relevant polyhedral product starting at some fixed basepoint.  

For any set $V$, let $\Delta_V$ denote the `simplex' with vertex set $V$, i.e., 
$\Delta_V$ is the simplicial complex whose vertex set is $V$ and whose simplices
are all of the finite subsets of $V$.  Thus $\Delta_V$ is a contractible complex
whose vertex set is $V$.  This construction is a functor from sets and functions 
to simplicial complexes and simplicial maps; in particular there is a natural 
bijection between the set of simplicial maps $\Delta_V\rightarrow \Delta_{V'}$ and 
the set of functions $V\rightarrow V'$.  

Now consider the case when $V=G$, a group.  In this case $\Delta_G$ is a contractible 
complex with $G$ as its set of vertices and with a natural action of $G$ that extends 
the standard action of $G$ on itself.  For each $i$, there is a homotopy equivalence 
of pairs $(EG_i,G_i)\rightarrow (\Delta_{G_i},G_i)$ that extends the identity map 
on $G_i$, and so by Proposition~\ref{prop:functorial}, the polyhedral products 
$(\Delta_{\underline{G}},\underline{G})^L$ and $(E\underline{G},\underline{G})^L$
are homotopy equivalent.  The $0$-skeleton of $(\Delta_{\underline{G}},\underline{G})^L$ 
is the set $\prod_iG_i$.  For any basepoint $\mathbf{g}\in \prod_iG_i$, the 
discussion above gives an identification between $G_\Gamma$ and the set of 
homotopy classes of paths from $\mathbf{g}$ to some element of $\prod_iG_i$ 
in $(\Delta_{\underline{G}},\underline{G})^L$, with the subgroup $K_\Gamma$ 
identified with the homotopy classes of loops based at $\mathbf{g}$. 

Since each $\Delta_{G_i}$ is a simplicial complex, the polyhedral product 
$(\Delta_{\underline{G}},\underline{G})^L$ is a polyhedral complex in which 
each polyhedron that arises is a product of simplices.  The polyhedral cells 
of $(\Delta_{\underline{G}},\underline{G})^L$ are easily described: the subsets
of the vertex set $\prod_iG_i$ that arise as the vertex set of some polyhedron 
are of the form $V_1\times V_2\times\cdots \times V_m$, where each $V_i$ is a 
non-empty finite subset of $G_i$, and if we define $\sigma\subseteq V(\Gamma)$
as $\sigma=\{i\in V(\Gamma)\,\,:\,\,|V_i|>1\}$, then $\sigma$ is either the 
empty set (in which case the polyhedron is a vertex) or $\sigma$ is a simplex 
of $L$.  This polyhedron is the product of simplices of dimensions $|V_i|-1$.  
It is apparent from this description that the group structure of $G_i$ is 
irrelevant to the structure of the polyhedral product 
$(\Delta_{\underline{G}},\underline{G})^L$, and that the construction of 
this polyhedral product can be viewed as a functor from ordered $m$-tuples 
of \emph{sets} and set maps to polyhedral complexes and polyhedral maps.  
The map $\mathbf{f}:\prod_iG_i\rightarrow \prod_i G'_i$ is the map on 
vertex sets, and so a polyhedron $V_1\times V_2\times \cdots \times V_m$ in 
$(\Delta_{\underline{G}},\underline{G})^L$ is sent to the polyhedron 
$(f_1(V_1)\times f_2(V_2)\times \cdots \times f_m(V_m))$ in 
$(\Delta_{\underline{G}},\underline{G})$.  To study fundamental groups
of course we only need to understand the 2-skeleta of these polyhedral 
products, so we give more detailed description of these.  

Any 2-cell in a product of simplicial complexes is either a triangle 
or a square, i.e., either a 2-simplex or a product of two 1-simplices.  
If $\mathbf{g},\mathbf{h}\in \prod_iG_i$ are a pair of distinct vertices of 
the polyhedral product, there is an edge between them if and only if 
there is a single $i\in \{1,\ldots,m\}$ for which $g_i\neq h_i$ or 
equivalently \[\mathbf{h}=(h_1,\ldots,h_m)=(g_1,\ldots,g_{i-1},h_i,g_{i+1}\ldots,g_m).\] 
Similarly, distinct vertices $\mathbf{g},\mathbf{h},\mathbf{k}$ span a 2-simplex 
if and only if there is a single $i$ for which $g_i,h_i,k_i$ are all distinct, 
while $g_j=h_j=k_j$ for each $j\neq i$.  Squares are most easily described by 
a long diagonal: if $\mathbf{g}$ and $\mathbf{h}$ are two sequences with the 
property that $g_k=h_k$ for each $k$ except that $g_i\neq h_i$ and $g_j\neq h_j$ 
for some $i<j$ such that $\{i,j\}$ is a 1-simplex of $L$ (or equivalently an edge of 
$\Gamma$), there is a square whose four vertices are 
$(g_1,\ldots,g_i,\ldots,g_j,\ldots g_m)$, 
$(g_1,\ldots,h_i,\ldots,g_j,\ldots g_m)$, 
$(g_1,\ldots,g_i,\ldots,h_j,\ldots g_m)$ and 
$(g_1,\ldots,h_i,\ldots,h_j,\ldots g_m)$.  

To find the connection between edge paths and the set $S^*$ of words in $S$, we need 
to \emph{label} the directed edges of $(\Delta_{\underline{G}},\underline{G})^L$ 
by elements of $S$.  The group $\prod_iG_i$ acts naturally on $(\Delta_{\underline{G}},\underline{G})^L$  via its free transitive action on the vertex set, and this action 
is compatible with the action of $\prod_iG_i$ on $(E\underline{G},\underline{G})^L$ 
by deck transformations, in the sense that any $G_i$-equivariant maps 
$(EG_i,G_i)\rightarrow (\Delta_{G_i},G_i)$ together give a $\prod_iG_i$-equivariant 
map $(E\underline{G},\underline{G})^L\rightarrow (\Delta_{\underline{G}},\underline{G})^L$.  
As noted above, two vertices $\mathbf{g}$ and $\mathbf{h}$ form an edge if and only if 
there is a single $i$ such that $g_i\neq h_i$.  Thus the element 
$\mathbf{g}^{-1}\mathbf{h}$ is of the form $(e_1,\ldots,s_i,\ldots e_m)$ for this same 
$i$, and so this element lies in $S$.  Thus every directed edge (i.e., every ordered 
pair $(\mathbf{g},\mathbf{h})$ such that $\{\mathbf{g},\mathbf{h}\}$ is an edge) is 
expressible in the form $(\mathbf{g},\mathbf{g}s)$ for some $s\in S$.  This gives a
$\prod_iG_i$-equivariant map from the directed edges of 
$(\Delta_{\underline{G}},\underline{G})^L$  to $S$, with the property that for any 
vertex $\mathbf{g}$, each label $s$ appears on exactly one directed edge starting 
at $\mathbf{g}$.  Hence the 1-skeleton of $(\Delta_{\underline{G}},\underline{G})^L$ 
is identified with the Cayley graph of $\prod_iG_i$ with respect to the generating 
set $S$.  If we fix a base vertex, such as $\mathbf{e}\in \prod_iG_i$, we thus 
get an identification of the edge paths starting at $\mathbf{e}$ with the set 
$S^*$ of words in $S$.  Passing from edge paths to homotopy classes  (relative to 
end points) of the edge paths corresponds to the natural map $S^*\rightarrow G_L$.  

Now that $S^*$, $G_L$, $K_L$ and $\prod_iG_i$ have all been identified in terms 
of $(\Delta_{\underline{G}},\underline{G})^L$ together with a choice of basepoint, 
we also get analogous descriptions for $S'^*$, $G'_L$, $K'_L$, and $\prod_iG'_i$.  
The maps $\Phi$ and $\Phi_{\G}$  appearing in the statement are the functions 
induced by $\mathbf{f}:(\Delta_{\underline{G}},\underline{G})^L\rightarrow 
(\Delta_{\underline{G'}},\underline{G'})^L$,  and it remains to compute them.  
Our base vertex $\mathbf{e}$ will map to some vertex $\mathbf{g}'$.  A directed 
edge from $\mathbf{g}$ to $\mathbf{g}s$ will map to the directed edge from 
$\mathbf{f}(\mathbf{g})$ to $\mathbf{f}(\mathbf{g}s)$, provided that these two 
vertices are distinct, and the label on this edge will be 
$\mathbf{f}(\mathbf{g})^{-1}\mathbf{f}(\mathbf{g}s)$, which is an element of $S'$. 
If on the other hand $\mathbf{f}(\mathbf{g})=\mathbf{f}(\mathbf{g}s)$ then 
the edge collapses to a vertex.  In this way, one sees that the map $\Phi:S^*\rightarrow S'^*$
described in the statement is the induced map from the set of edge paths starting 
at some vertex $\mathbf{g}$ to the set of edge paths starting at $\mathbf{f}(\mathbf{g})$.  
The claim concerning $\Phi_{\G} :\GG\rightarrow \GG'$ follows immediately, as does the 
fact that $\Phi_{\G}$  is a homomorphism when restricted to the subgroup $K_L$ of closed
paths.  
\end{proof}

\begin{rmk}\rm As we have noted earlier, the polyhedral products 
$(\Delta_{\underline{G}},\underline{G})^L$ and $(E\underline{G},\underline{G})^L$
are homotopy equivalent. However,
the complex $\Delta_G$ used in the above argument is only a model for $EG$ in the 
case when $G$ is torsion-free.  If $G$ contains a non-trivial finite subgroup 
$H$ the action is no longer free: the barycentre of each subsimplex 
$\Delta_{Hg}\subseteq \Delta_G$ with vertex set a coset $Hg$ is fixed by the 
action of the subgroup $H$.  
In the general case, $\Delta_G$ (or rather its barycentric subdivision, 
the simplicial complex arising as the realization of the poset of finite 
subsets of $G$) is a model for $\underline{E}G$, the universal proper $G$-CW-complex (for definition, see Section \ref{sec:Brown}).  
\par
It is possible to rephrase the argument using the pair $(E_G,G)$ instead of 
$(\Delta_G,G)$, where $E_G$ denotes Segal's simplicial set model for $EG$, which 
is the universal cover of Segal's model $B_G$ for $BG$.  The geometry here becomes 
a little more complicated, which is why we prefer $\Delta_G$.  The non-degenerate 
$k$-simplices of $E_G$ are given by ordered $(k+1)$-tuples $(g_0,g_1,\ldots,g_k)$
such that for each $i<k$, $g_i\neq g_{i+1}$.  Just as for $\Delta_G$, the group 
structure on $G$ is irrelevant to the construction of $E_G$ and so it may be 
viewed as a functor from sets to simplicial sets.  Since the vertices of each 
simplex come with an ordering, each edge has its own preferred orientation, 
and so there are two distinct edge paths of length one from $g$ to $h$: travel 
along the edge $(g,h)$ in its preferred orientation, or travel along the edge 
$(h,g)$ in its opposite orientation.  The `2-simplex' $(g,h,g)$ is a triangle
in which the side corresponding to the subsequence $(g,g)$ collapses to a point, 
so this 2-simplex provides a homotopy between these two length one edge paths. 
The proof goes fairly similarly, except that now $S^*$ is identified with the 
set of \emph{directed edge paths}, i.e., edge paths in which edges are only 
travelled in their preferred direction.  Every sequence of adjacent vertices 
arises from exactly one directed edge path and any edge path is homotopic relative
to its vertex set to a directed edge path. 
\end{rmk}

\begin{proof}[Proof of Proposition~\ref{prop:explicit_homomorphism}]
It remains only to verify that the two descriptions of the group homomorphism 
$\Phi_\Gamma: K_\Gamma\rightarrow K'_\Gamma$ agree.  In fact the definition 
of $\Phi_\Gamma$ given in the statement of Proposition~\ref{prop:explicit_homomorphism}
extends by the same formula to a \emph{set map} $\Phi_\Gamma:G_\Gamma\rightarrow G'_\Gamma$, 
and we verify that the two definitions agree by induction on the length of the given 
syllable representation for $g\in G_\Gamma$.  Under both definitions, the identity element 
maps to the identity element.  Let $\pi:G_\Gamma\rightarrow \prod_iG_i$ be the projection 
from the graph product to the direct product.   If the two definitions agree on $g$, then the 
second definition gives that $\Phi_\Gamma(gs)=\Phi_\Gamma(g)\mathbf{f}(\pi(g))^{-1}\mathbf{f}(\pi(g)s)$, 
where we do not need to split into cases because we are working in the group rather than 
with words.  If $\pi(g)=\mathbf{g}=(g_1,\ldots,g_m)$ and $s\in G_i\cap S$ then 
$\pi(g)s=(g_1,\ldots,g_is,\ldots,g_m)$ and so 
\begin{align*}
\mathbf{f}(\pi(g))^{-1}&\mathbf{f}(\pi(g)s) \\
=&(f_1(g_1)^{-1},\ldots,f_i(g_i)^{-1},\ldots 
f_m(g_m)^{-1})(f_1(g_1),\ldots,f_i(g_is),\ldots,f_m(g_m)) \\ = &
(e_1,\ldots,e_{i-1},(f_i(g_i)^{-1}f_i(g_is)),e_{i+1},\ldots, e_m).  
\end{align*}
In this notation, $g_i$ is the ordered product of all the earlier syllables $g'_{ij}\in G_i$ 
for $w$ and $s=g_{ij}$, so the two expressions are equal.  
\end{proof}

Let $\G$ be a simplicial graph and let $\{G_i\}_{i \in V(\G)}$ be a collection of vertex groups. The \emph{length} of a word $W\in \GG$, denoted $l(W)$, is the number of syllables it contains.
The following \emph{elementary moves} transform a word $W\in \GG$ into another word $W'$ representing the same element of $G_\G$, without increasing its length:

\begin{enumerate}
    \item Remove a trivial syllable: if $g_j = e_i\in G_i$.
    \item Merge consecutive syllables in the same vertex group: if $g_j, g_{j+1} \in G_i$, replace them by $g_j g_{j+1}$.
    \item Swap commuting syllables: if $g_j \in G_i$, $g_{j+1} \in G_k$, and $\{i,k\}$ is an edge in $\G$, exchange $g_j$ and $g_{j+1}$.
\end{enumerate}

\begin{deff}\label{def:normal_form}\rm
A word $W$ representing $g \in G_\G$ is said to be in \emph{normal form} if no sequence of the above elementary moves can produce a strictly shorter word representing~$g$.
\end{deff}

 \begin{example}\label{ex:W_A} \rm We illustrate the formula of Proposition \ref{prop:explicit_homomorphism} in a case where $\GG$ is the right-angled Coxeter group. 
 Let $\G$ be an arbitrary simplicial graph with vertex set $V(\G)$, and let $W_{\G}$
be the corresponding right-angled Coxeter group. Let $\{G_i'\}_{i\in V(\G)}$ be any collection of groups, and let $f_i : C_2 \longrightarrow G_i'$
be arbitrary set maps.  

Let $W \in [W_\G, W_\G]$ be a word in normal form in $W_\G$, and for each $i \in V(\G)$, let $g_{ij}$, $1 \le j \le m_i$, denote the $j$-th occurrence of the nontrivial generator from $G_i$ in $W$.  

Then the induced homomorphism $\Phi: [W_\G, W_\G]\to K_\G'$ simplifies according to the parity of the syllable index $j$:
\[
g'_{ij} =
\begin{cases}
f_i(e_i)^{-1} \cdot f_i(g_{ij}), & \text{if $j$ is odd},\\[2mm]
f_i(g_{ij})^{-1} \cdot f_i(e_i), & \text{if $j$ is even}.
\end{cases}
\]
Thus, $\Phi(W)$ is obtained by replacing each syllable $g_{ij}$ in $W$ with $g'_{ij}$ in the same syllable position.

Moreover, let $\GG'$ be the right-angled Artin group associated to $\G$, and define $f_i: C_2 \to \Z$ by sending each generator of $C_2$ to a generator $x \in \Z$ and the identity of $C_2$ to $0 \in \Z$.  
Then the formula further simplifies to
\[
g'_{ij} =
\begin{cases}
x, & \text{if $j$ is odd},\\[1mm]
x^{-1}, & \text{if $j$ is even}.
\end{cases}
\]

Similarly, we can consider the epimorphisms $f_i': \Z \to C_2$.  
Let $\Psi: [A_\G, A_\G] \longrightarrow [W_\G, W_\G]$ denote the induced homomorphism on the commutator subgroups.  
Note that since $f_i' \circ f_i = \mathrm{id}$, it follows that $\Phi$ is a retract of $\Psi$, which can also be seen from the explicit formulas of $\Phi$ and $\Psi$.

\end{example}

\section{Some Applications}

\subsection{Poly-freeness} A group $G$ is {\it poly-free} if there exists a finite chain of normal subgroups
$$1\unlhd G_1 \unlhd \cdots \unlhd G_n=G$$
such that $G_{i+1}/G_i$ is a free group.

\begin{cor}\label{cor:poly-free_kernel}
    Let $\G$ be a finite simplicial graph and let $\GG$ be the graph product of countable groups $\{G_i\}_{i \in V(\G)}$. Then $\KG$ is poly-free.
\end{cor}
\begin{proof} By Corollary \ref{intro_cor}, it suffices to consider the case when $G_i=\Z$ and $\GG=A_{\G}$. If $\G$ as a complete graph, then $\KG=1$ and there is nothing to prove. Assume that $\G$ is not complete, and let $\overline{\G}$ be the graph obtained from $\G$ by adding a single edge. Proceeding by induction, we can assume that $K_{\bG}$ is poly-free. Then, as in the proof of Theorem \ref{thm:general_kernel_functoriality}, there is a short exact sequence of groups
$$K_{\G,\bG} \rightarrowtail K_{\G} \twoheadrightarrow K_{\bG}.$$ 
We claim that $K_{\G,\bG}$ is a free group. To see this, consider the short exact sequence:
$$K_{\G,\bG} \rightarrowtail G_{\G} \twoheadrightarrow G_{\bG}.$$ 
Let $v\in V(\G)$ adjacent to the added edge in $\bG$. Let $\G'$ be the full subgraph of $\G$ (and hence of $\bG$) on the vertex set $V(\G)\smallsetminus v$. Denote by $N$ the kernel of the standard epimorphism from $\GG$ to $G_{\G'}$.
Since this epimorphism factors through $G_{\bG}$, it follows that $K_{\G,\bG}$ is a subgroup of $N$. But $N$ is a free group, as noted in the proof of Lemma 1 of \cite{Howie99}. Hence, $K_{\G,\bG}$ is also free.
\end{proof}

\begin{rmk}\rm Alternatively to the above argument, one may note that by the proof of Lemma 1 of \cite{Howie99} (see also \cite{HerSun07}), $A_{\G}$ is poly-free and hence, so is $[A_{\G}, A_{\G}]$.  
\end{rmk}

As an immediate application, we obtain:

\begin{cor}\label{cor:poly-free_graph_prodcut}
    Let $\G$ be a finite simplicial graph and let $\GG$ be the graph product of countable poly-free groups $\{G_i\}_{i \in V(\G)}$. Then, $\GG$ is poly-free.
\end{cor}

\subsection{Normal length metric} 

  We recall the normal form for elements of a graph product and the associated
syllable length, which induces a canonical left-invariant metric on the kernel
$K_{\G}$.

\begin{deff}\label{def:normal_form}\rm
 We define the {\it normal length}, denoted $nl(g)$, of $g\in G_{\G}$ as the length of the word $W$ in normal form that represents $g$. 
\end{deff}

The normal length of an element $g$ in $\GG$ is well-defined because of the uniqueness (up to permutation of syllables) of the word in normal form representing $g$ \cite[Theorem 4.8]{Green1991}.  

\begin{deff}\label{def:normal_metric}\rm
The \emph{normal length metric} on $G_{\G}$ is the left-invariant metric
$d_{nl}$ defined by
\[
d_{nl}(g,h) := nl(g^{-1}h), \quad g,h \in G_{\G}.
\]
We use the same notation for its restriction to the kernel $K_{\G}$.
\end{deff}

\begin{rmk}\label{remark:word_length}\rm
In the case of a right-angled Coxeter group $W_{\G}$, the normal length and its induced metric are precisely the same as the word length and the word metric derived from the standard generating set of $W_{\G}$. We note that the growth series of $W_{\G}$ with respect to the word metric is explicitly computed, e.g.\,\cite[Chapter 17]{Davis_book}.
\end{rmk}

From Proposition \ref{prop:explicit_homomorphism}, we obtain:

\begin{cor}\label{lem:normal_form}
    Let $\G$ be a simplicial graph, and let $\{G_i\}_{i \in V(\G)}$ and $\{G_i'\}_{i \in V(\G)}$ be two families of groups. Given a collection of injective set maps $f_i : G_i \longrightarrow G_i'$, suppose
$\Phi : K_{\G} \longrightarrow K_{\G}'$  is the induced homomorphism. Then, for any $g\in \KG$, we have 
$$nl(g)= nl(\Phi(g)).$$
\end{cor}

\begin{proof} We use the notation of Proposition \ref{prop:explicit_homomorphism}. Let $W\in K_{\G}$ be a word in a normal form representing $g\in \GG$.  
By Proposition \ref{prop:explicit_homomorphism}, $\Phi(g)$ is represented by the word $W'\in K_{\G'}$ whose syllables are 
\[
g'_{ij} = f_i(g_{i0} \cdots g_{ij-1})^{-1} \cdot f_i(g_{i0} \cdots g_{ij}) \in G_i' \mbox{ and } g'_{i0}=e'_i,
\]
appearing in the same syllable positions as in $W$.  
It suffices to show that $W'$ is in normal form.

Each $g'_{ij} \neq e_i'$, for otherwise
\[
f_i(g_{i0} \cdots g_{ij-1})^{-1} \cdot f_i(g_{i0} \cdots g_{ij}) = e_i',
\]
which would contradict the injectivity of $f_i$. Hence, no syllable in $W'$ can be removed.

Suppose $W'$ could be shortened by a sequence of commuting moves. Then, for some $i \in V(\G)$, there exist $1 \le p \le q \le m_i$ such that the syllables $g'_{ip}$ and $g'_{iq}$ can be made consecutive. This is because any syllables occurring between them in $W'$ must come from vertex groups adjacent to $i$, and thus can be swapped past $g'_{ip}$ and $g'_{iq}$. The corresponding syllables in $W$ could then also be brought together in the same way, producing a shorter word than $W$, contradicting its normal form.

\end{proof}

\begin{cor}\label{cor:isometric_embedding}
Under the assumptions of Corollary~\ref{lem:normal_form}, the induced homomorphism
\[
\Phi : (K_{\G}, d_{nl}) \longrightarrow (K_{\G}', d_{nl})
\]
is an isometric embedding. 
\end{cor}

\begin{proof}
For all $g,h \in K_{\G}$,
\[
d_{nl}(\Phi(g),\Phi(h)) = nl(\Phi(g^{-1}h)) = nl(g^{-1}h) = d_{nl}(g,h),
\]
where the second equality follows from Corollary~\ref{lem:normal_form}.
\end{proof}

\begin{cor}\label{cor:isometric_RAAG}
Let $\G$ be a simplicial graph and suppose that each vertex group $G_i$ is
nontrivial and countable.
Then the kernel $K_{\G}$ is isometrically isomorphic to the commutator subgroup
$[A_{\G},A_{\G}]$ of the right-angled Artin group on $\G$, with respect to the
normal length metric.
\end{cor}

 \subsection{Hereditary properties} Let  $\C$ be a class of groups satisfying the following  properties:

\begin{enumerate}[(P1)]
    \item Every subgroup of any finitely generated right-angled Artin group is in $\C$;
    \item If $B$ is a group that has a normal subgroup $A\in \C$ such that $B/A \in \C$, then $B\in \C$.
    
\end{enumerate}

\begin{cor}\label{general_property} Let $\C$ be a class of groups satisfying properties $(P1)$ and  $(P2)$.
Let $\G$ be a finite simplicial graph and let $\GG$ be the graph product of countable groups $\{G_i\}_{i \in V(\G)}$.  If each $G_i\in \C$, then $\GG\in \C$.
\end{cor}
\begin{proof} Suppose that each $G_i\in \C$. By Corollary \ref{intro_cor},  $\KG$ is a subgroup of  $[A_\G, A_\G]$ and hence, is in $\C$ by (P1). 
   We have an extension: 
$$1\to K_\G \to \GG\to \prod_{i \in V(\G)} G_i\to 1.$$
Applying (P2), it follows that $\GG\in \C$. 
\end{proof}

Our first application of Corollary \ref{general_property} is to the Baum–Connes Conjecture with coefficients with finite wreath products.

\begin{cor}\label{cor:Baum--Connes}
Let $\G$ be a countable simplicial graph  and let $\GG$ be the graph product of countable groups $\{G_i\}_{i \in V(\G)}$.  
If each vertex group $G_i$ satisfies BCC with finite wreath products, then $\GG$ also satisfies BCC with finite wreath products.
\end{cor}

\begin{proof} First, assume $\G$ is finite.  Let $\C$ be the class of groups satisfying BCC with finite wreath products. Recall that any right-angled Artin group satisfies  BCC with finite wreath products as it is a-T-amenable  \cite{NR97, HK01}.
Since BCC with finite wreath products is subgroup and extension closed property \cite[Theorem 4.3(2), (6)]{NP25},  it follows that $\C$ satisfies properties (P1) and (P2).
Therefore, by Corollary \ref{general_property}, $\GG\in \C$. 

In the case when $\G$ is infinite, then $\GG$ is an increasing countable union of subgroups that are graph products of an increasing union of finite full subgraphs of $\G$. The result then follows from \cite[Theorem 4.3(7)]{NP25}. 
\end{proof}

Recall, a group $G$ is said to be {\it locally indicable}, if every finitely generated subgroup of $G$ admits an epimorphism to $\Z$. We are able to deduce the following special case of \cite[Corollary 5.8]{Antolin_Minasyan_2015}.

\begin{cor}\label{cor:loc_indicable}
Let $\G$ be a finite simplicial graph  and let $\GG$ be the graph product of countable groups $\{G_i\}_{i \in V(\G)}$.  
If each vertex group $G_i$ is locally indicable, then $\GG$ is locally indicable.
\end{cor}
\begin{proof} Let $\C$ be the class of all countable locally indicable groups. It is easy to see that $\C$ satisfies properties (P2).

By Theorem 2.18 of \cite{Green1991}, any finitely generated right-angled Artin group is locally indicable. Since local indicability is a subgroup-closed property, by Corollary \ref{general_property}, $\GG\in \C$. 
\end{proof}

Recall, a group $G$ is said to satisfy {\it Tits Alternative}, if either $G$ is virtually solvable or it contain a non-abelian free subgroup. We are able to deduce the special case of \cite[Theorem A]{Antolin_Minasyan_2015}.

\begin{cor}\label{cor:Tits_alternative}
Let $\G$ be a finite simplicial graph  and let $\GG$ be the graph product of countable groups $\{G_i\}_{i \in V(\G)}$.  
If each vertex group $G_i$ satisfies Tits Alternative, then $\GG$ does as well.
\end{cor}
\begin{proof} Let $\C$ be the class of all countable groups that are either virtually solvable or contain a non-abelian free subgroup. It is easy to see that $\C$ satisfies properties (P2), since any extension of a virtually solvable group by a virtually solvable group is again virtually solvable.

By Corollary 3.6 of \cite{Hsu_Wise_99}, any finitely generated right-angled Artin group is linear over $\Z$. Hence, any of its subgroups are either virtually polycyclic or contain a non-abelian free subgroup. So, by Corollary \ref{general_property}, $\GG\in \C$.    
\end{proof}

\subsection{Cohomological dimension.}

Before stating the next corollary, we recall the definition of the Bestvina--Brady group associated to a flag complex.  
Let $L$ be a finite flag complex, and let $A_L$ denote the right-angled Artin group on the 1-skeleton of $L$. Consider the canonical homomorphism
\[
\phi: A_L \longrightarrow \mathbb{Z}, \quad v \mapsto 1 \text{ for each vertex } v \in V(L),
\]
which sends each standard generator of $A_L$ to $1 \in \mathbb{Z}$. The \emph{Bestvina--Brady group} $B_L$ is defined as the kernel of this homomorphism:
\[
 B_L := \ker(\phi) \triangleleft A_L.
\]

\begin{cor}\label{cor_vcd}
Let $L$ be an $n$-dimensional finite flag complex and let $G_L$ be the graph product of nontrivial countable groups $\{G_i\}_{i \in V(L)}$. 
\begin{enumerate}[(i)]
    \item \label{part 1} If $H^n(L; \Z)\ne 0$, then $\cd K_L = \vcd W_L =\cd B_L=\cd A_L=n+1$.
   \item \label{part 2} If $H^n(L; \Z)=0$, then $\cd K_L\leq n$. 
In addition, if there is a full subcomplex $Q\subseteq L$ such that $H^{n-1}(Q; \Z)\ne 0$, then $\cd K_L = \vcd W_L =\cd B_L=n$.

   \end{enumerate}
\end{cor}

\begin{proof} By Corollary \ref{intro_cor}, $[W_L, W_L]\leq K_L\leq [A_L, A_L]$. Clearly, $[A_L, A_L]\leq B_L\leq A_L$. 

 If $H^n(L; \Z)\ne 0$, from Corollary 8.5.5 of \cite{Davis_book}, it follows the $\vcd W_L=\cd [W_L, W_L]=n+1$. On the other hand, $\cd A_L=n+1$. The equalities in (\ref{part 1}) now follow by a comparison of dimensions.

 If $H^n(L; \Z)=0$, then by Theorem 22 of \cite{LS2011}, $\cd B_L\leq n$. Thus, $\cd K_L\leq n$.
 
 Suppose, in addition that there is a full subcomplex $Q\subseteq L$ such that $H^{n-1}(Q; \Z)\ne 0$, then by Corollary 8.5.5 of \cite{Davis_book}, $\vcd W_L \geq \vcd W_Q\geq n$.  The equalities in (\ref{part 2}) then follow by a comparison of dimensions.
 \end{proof}

 \subsection{Brown's question}\label{sec:Brown} A model for $\uE G$ is proper $G$-CW complex $X$ such that for any finite $F\leq G$, the fixed point set $X^F$ is contractible. Such an $X$ always exists and is unique up to $G$-homotopy. Let $\ugd G$ denote the minimal dimension of any model for $\uE G$. 
 
 One of Brown's questions \cite[Question 2, page 32]{brown_1979}, see also \cite[Question 2]{CK_1986} and \cite[VIII.11]{brown1982cohomology}, asks whether $\ugd G=\vcd G$ for any virtually torsion-free group $G$ that admits a cocompact model for $\uE G$. In \cite{LP_2017}, we gave the first counterexamples to this question. These groups were finite extensions of the commutator subgroup of certain right-angled Coxeter groups. The next application shows that similar counterexamples exist for more general graph products. 

 \begin{cor}\label{cor:Brown_question} For $n\geq 2$,
let $L$ be the $n$-dimensional flag complex determined by a finite simplicial graph $\G$. Let $H$ be a countable nontrivial group and denote by $\GG$  the graph product of  $\{G_i\}_{i \in V(\G)}$ such that each $G_i=H$. Suppose $Q$ is a finite subgroup of automorphisms of $\G$ and $G$ be the corresponding semidirect product $\KG\rtimes Q$. Denote by $L^{\rm sing}$ the subcomplex of $L$ consisting of points with nontrivial stabilizer in $Q$. If $H^n(L)=0$ and $H^n(L, L^{\rm sing})\ne 0$, then
$$\vcd G\leq n \;\;\; \mbox{and}\;\;\; \ugd G=n+1.$$
Furthermore, $\vcd G=n$, if either $L$ is a barycentric subdivision or $L^{\rm sing}$ is a full subcomplex of $L$.
\end{cor}
\begin{proof}  By Corollary \ref{intro_cor}(iii), $[W_L, W_L]$ is a retract of $K_L$. Denote by $G'$ the corresponding semidirect product $[W_L, W_L]\rtimes Q$. It follows that $G'$ is a retract of $G$. By Theorem 1.1 of \cite{LP_2017}, $\ugd G'=n+1$. Therefore, $\ugd G\geq n+1$.

To show the opposite inequality, we recall that $G_L$ acts on the Davis building $X$ with stabilisers precisely the conjugates of the vertex group $G_i$ \cite{Davis_98}. This building is a CAT(0) cube complex where cubes of maximal dimension are homeomorphic to the cone on the first barycentric subdivision of $L$. Therefore, $X$ has dimension $n+1$. The action of $Q$ on $L$ combines with action of $G_L$ on $X$ to give an action of $G_L\rtimes Q$ on $X$. This is evident from the panel structure of $X$ (see for example \cite[II.12]{bridson2013metric} and also \cite[\S 3]{PP_24}). Since $K_L$ acts freely on $X$, the restriction of this action to  $G$ is proper. Since $X$ is CAT(0), it follows that it is a model for $\uE G$ of dimension $n+1$. Therefore, $\ugd G= n+1$. 

By Corollary \ref{cor_vcd}, $\vcd G\leq n$. If $L$ is a barycentric subdivision or $L^{\rm sing}$ is a full subcomplex of $L$, then by Theorem 1.1 of \cite{LP_2017}, $\vcd G'=n+1$, which implies that $\vcd G=n+1$.    
\end{proof}

\begin{example}\label{example_1}\rm
Let $Q = A_5$, and let $L$ be the acyclic 2-dimensional flag complex from \cite[Example 1]{LP_2017}, with a simplicial action of $Q$. 
Denote by $L^{\rm sing}$ the subcomplex consisting of points with nontrivial $Q$-stabilizer. Then, $L$ satisfies  $H^2(L, L^{\rm sing}) \neq 0.$

For any countable nontrivial group $H$, let $\GG$ be the graph product of copies of $H$ over the graph $\Gamma$ determined by $L$, and set $G := \KG \rtimes Q$.  
Then, by Corollary \ref{cor:Brown_question}, $\vcd G \le 2$ and $\ugd G = 3$. 
Furthermore, since $L$ is the barycentric subdivision of a polygonal complex, we have $\vcd G = 2$.
\end{example}

\begin{example}\label{example_2}\rm
Fix distinct primes $p$ and $q$, and let $Q$ be cyclic of order $q$. Let $L$ be a contractible $Q$-CW-complex constructed as in \cite[Example 2]{LP_2017}, with $L^Q = L^{\rm sing}=M(1,p)$ mod-$p$ Moore space. Then, $H^3(L, L^{\rm sing}) \neq 0$. 
For any countable nontrivial group $H$, let $\GG$ be the graph product of copies of $H$ over the corresponding graph, and set $G := \KG \rtimes Q$.  
Then, by Corollary \ref{cor:Brown_question}, $\vcd G \le 3$ and $\ugd G = 4$. 
\end{example}

\begin{rmk} \rm When $H$ is finite, both examples exhibit groups $G$ that admit cocompact models for $\uE G$. This follows from the fact that $G_L\rtimes Q$ acts cocompactly on the Davis building $X$ and and since $G$ is a finite index subgroup, so does $G$.
\end{rmk}

\begin{rmk} \rm Just as in \cite[Theorem 1.1]{LP_2017}, by taking a direct product of copies of $G$ satisfying the conclusion of Corollary~\ref{cor:Brown_question}, one obtains that
$$\vcd G^m\leq m\cdot n \;\;\; \mbox{and}\;\;\; \ugd G^m=m\cdot n+m, \;\; \forall m\in \mathbb N.$$

\end{rmk}

\subsection{Eilenberg--Ganea conjecture} A well-known conjecture of Eilenberg and Ganea predicts the following.
\begin{conjecture}[Eilenberg--Ganea, 1957] If a group $G$ has $\cd G=2$, then there exists a $2$-dimensional $K(G,1)$-complex. 
\end{conjecture}

 In  \cite[Theorem~8.7]{BB_1997}, Bestvina and Brady  proved that if $L$ is a flag triangulation of a spine of the Poincar\'{e} homology sphere (see also Example \ref{example_1}), then $B_L$ is a counterexample to either the Eilenberg--Ganea Conjecture or the Whitehead Conjecture \cite {whitehead_41}.

 In \cite{LP_2017}, we showed that for the same acyclic flag $2$-complex $L$, $\vcd W_L=2$ but $W_L$ does not act properly and cellularly on any contractible $2$-complex. If $[W_L, W_L]$ was a counterexample to the Eilenberg--Ganea Conjecture, then it would also be a counterexample to Wall's D(2) Problem \cite[D3]{wall_d2}.

In his discussion of the Eilenberg--Ganea Conjecture, Davis remarks that the 
 Bestvina--Brady group \(B_L\) is a ``more convincing candidate  for a counterexample'' than  \([W_L,W_L]\) 
(see \cite[p.~156]{Davis_book}). The following formalizes this relation.

 \begin{cor} Let $L$ be a flag triangulation of a spine of the Poincar\'{e} homology sphere. If $[W_L, W_L]$ is a counterexample to the Eilenberg--Ganea Conjecture, then $B_L$ is  a counterexample as well.
 \end{cor}
 \begin{proof}By Corollary \ref{intro_cor} (iii), $[W_L, W_L]$  embeds as a subgroup in $[A_L, A_L]$ which in turn is a subgroup of $B_L$. Since $\cd [W_L, W_L]= \cd B_L=2$, the claim follows.
     
 \end{proof}

\bibliographystyle{alpha}
\bibliography{petrosyan_refs}

\end{document}